\documentclass[12pt,reqno]{amsart}
\usepackage{hyperref,cleveref,amssymb,eqnarray,mathrsfs,xcolor}

\theoremstyle{plain}
\newtheorem{theorem}{Theorem}[section]

\newtheorem{corollary}[theorem]{Corollary}
\newtheorem{lemma}[theorem]{Lemma}

\theoremstyle{definition}

\newcommand{\Pelc}[1]{\left(\Sigma X_n\right)_p}

\DeclareSymbolFont{bbold}{U}{bbold}{m}{n}
\DeclareSymbolFontAlphabet{\mathbbold}{bbold}

\renewcommand{\le}{\leqslant}
\renewcommand{\ge}{\geqslant}

\begin{document}

\baselineskip 7mm

\title{On commutators of idempotents} 

\author{R. Drnov\v sek}
\address{Faculty of Mathematics and Physics, University of Ljubljana,
  Jadranska 19, 1000 Ljubljana, Slovenia \ \ \ and \ \ \
  Institute of Mathematics, Physics, and Mechanics, Jadranska 19, 1000 Ljubljana, Slovenia}
\email{roman.drnovsek@fmf.uni-lj.si}


\date{\today}

\begin{abstract}
Let  $T$ be an operator on Banach space $X$ that is similar to $- T$ via an involution $U$. 
Then $U$ decomposes the Banach space $X$ as $X = X_1 \oplus X_2$ with respect to which decomposition we have 
$U = \left(\begin{matrix} I_1 & 0 \\  0 & -I_2 \end{matrix} \right)$, where $I_i$ is the identity operator on the closed subspace $X_i$ ($i=1, 2$).
Furthermore, $T$ has necessarily the form 
$T = \left(\begin{matrix} 0 & * \\  * & 0 \end{matrix} \right) $ with respect to the same decomposition. In this note we consider the question when $T$ 
is a commutator of the idempotent $P = \left(\begin{matrix} I_1 & 0 \\  0 & 0 \end{matrix} \right)$ and some idempotent $Q$ on $X$.
We also determine which scalar multiples of unilateral shifts on $l^p$ spaces  ($1 \le p \le \infty$)
are commutators of idempotent operators.
\end{abstract}

\maketitle

\section{Introduction}

A (bounded linear) operator $C$ on a Banach space is said to be the {\it commutator} of the 
operators $A$ and $B$ if $C = A B - B A = [A, B]$. The operators on a Hilbert space
that arise as commutators have been characterized by Brown and Pearcy 
\cite{BP65} as the operators that are not the sum of a compact operator
and a non-zero scalar multiple of the identity.

It is natural to ask which operators are commutators of operators of given forms.
For example, commutators of self-adjoint operators have been characterized in \cite{Ra66},  and commutators of idempotent matrices have been characterized in \cite{DRR02}.
In this paper we first improve the  ring-theoretic characterization of commutators of idempotents from  \cite{DRR02}, and then we characterize  commutators of the idempotent operators on a Banach space.
Motivated with the case $p=2$ in \cite[Corollary 5.9]{MRZ23} we also determine which scalar multiples of unilateral shifts on $l^p$ spaces  ($1 \le p \le \infty$)
are commutators of idempotent operators.

We now recall some definitions. Let $R$ be a unital ring with identity $1$. An {\it idempotent} is any $p \in R$ such that $p^2 = p$, while $u \in R$ is called an {\it involution} if $u^2 = 1$.
Of course, elements in $R$ of the form $[a, b] := a b - b a$ are called {\it commutators}. \\

\section{A characterization of commutators of idempotents}

We first complement the ring-theoretic characterization of commutators of idempotents as given in \cite[Theorem 1]{DRR02}. This sheds new light on the proof of  \cite[Theorem 1]{DRR02}.

\begin{theorem}
Let $R$ be a unital ring with identity $1$ in which the element $2$ is invertible, and let $t \in R$. The following assertions are equivalent:
\begin{enumerate}
\item $t$ is a commutator of a pair of idempotents in $R$;

\item $4 t$ is a commutator of a pair of involutions in $R$;

\item There exist $u \in R$ and $s \in R$ such that $u^2 = 1$, $u t + t u = 0$, 
$u s = s u$, $s t = t s$ and $s^2 = t^2 + 1/4$.
\end{enumerate}
\label{ring}
\end{theorem}

\begin{proof}
Assume that $t = p q - q p$ for some idempotents $p$ and $q$. Define 
$u := 2 p - 1$ and $v := 2 q - 1$. Then $u^2 = v^2= 1$ and 
$$ 4 t = [2p, 2q] = [2p-1,2q] = [u, 2q-1] = [u, v] . $$
This proves the implication (i) $\Rightarrow$ (ii).
Since the proof of the converse implication is similar, we omit it.

Now assume (ii), that is, $4 t = u v - v u$ for some involutions $u$ and $v$. Then 
$$ 4 u  t + 4 t u =  (v - u v u) + (u v u - v) = 0  . $$
Define 
$$ s := \frac{1}{4} (u v + v u)  . $$ 
Then
$$ 4 u s = v + u v u = 4 s u , $$
$$ 16 s t = (u v + v u)(u v - v u) = (u v)^2 - 1 + 1 - (v u)^2 = 16 t s, $$ 
and 
$$ 16 s^2 - 16 t^2 =((u v)^2 + 1+1 + (v u)^2) - ((u v)^2 - 1-1 + (v u)^2) = 4 . $$
This completes the proof of the implication (ii) $\Rightarrow$ (iii).

Now assume (iii). Define $v := 2 u (s+t) = 2 (s - t) u$. Then 
$$ v^2 = 2 (s-t) u \cdot 2 u (s+t) = 4(s^2 - t^2) = 1 , $$
and 
$$ u v - v u = 2(s+t)-2(s-t) = 4 t . $$
This proves (ii), and it completes the proof of the theorem.
\end{proof}

Let $T$ be an operator on a complex Banach space $X$. Suppose that $T$ is similar to $- T$ via an involution $U$, so that $T U + U T = 0$.
Then $U$ decomposes the Banach space $X$ as $X = X_1 \oplus X_2$ with respect to which decomposition we have 
$$ U = \left(\begin{matrix} I_1 & 0 \\  0 & -I_2 \end{matrix} \right) , $$
where $I_i$ is the identity operator on the closed subspace $X_i$ ($i=1, 2$).
Write $$ T = \left(\begin{matrix} A & B \\  C & D \end{matrix} \right)  $$
with respect to the same decomposition. 
Since  $T U + U T = 0$, we have $A = 0$ and $D= 0$. Now \Cref{ring} gives the following characterization.

\begin{theorem}
The operator 
 $$ T = \left(\begin{matrix} 0 & B \\  C & 0 \end{matrix} \right) $$
is a commutator of the idempotent $P = \left(\begin{matrix} I_1 & 0 \\  0 & 0 \end{matrix} \right)$ and some idempotent $Q$ on $X$ if and only if 
the operator $B C + \frac14 I_1$ has a square root $S_1$ on $X_1$, 
the operator $C B + \frac14 I_2$ has a square root $S_2$ on $X_2$, and $S_1 B = B S_2$, $S_2 C = C S_1$. 
\label{operators}
\end{theorem}

\begin{proof}
($\Rightarrow$)
 By \Cref{ring} and its proof, there exists an operator $S$ on $X$ such that $S U = U S$, $S T = T S$ and $S^2= T^2 + \frac14 I$. 
Since $S U = U S$, $S$ has the form $S = \left(\begin{matrix} S_1 & 0 \\  0 & S_2 \end{matrix} \right)$. 
Since 
$$ \left(\begin{matrix} S_1^2 & 0 \\  0 & S_2^2 \end{matrix} \right) = S^2 = T^2+ \frac14 I = \left(\begin{matrix} B C + \frac14 I_1 & 0 \\  0 & C B+ \frac14 I_2 \end{matrix} \right) , $$
we have $S_1^2 = B C + \frac14 I_1$ and $S_2^2 = C B+ \frac14 I_2$. It follows from $S T = T S$ that $S_1 B = B S_2$ and $S_2 C = C S_1$. 

($\Leftarrow$)
Define $S := \left(\begin{matrix} S_1 & 0 \\  0 & S_2 \end{matrix} \right)$ on $X = X_1 \oplus X_2$. Then 
$$ S^2 = \left(\begin{matrix} S_1^2 & 0 \\  0 & S_2^2 \end{matrix} \right) = \left(\begin{matrix} B C + \frac14 I_1 & 0 \\  0 & C B+ \frac14 I_2 \end{matrix} \right) = T^2+ \frac14 I  . $$
Since $S_1 B = B S_2$, $S_2 C = C S_1$, we have $S T = T S$. Clearly, $S U = U S$. By \Cref{ring} and its proof,
$T = \left(\begin{matrix} 0 & B \\  C & 0 \end{matrix} \right) $  is a commutator of the idempotent $P = \left(\begin{matrix} I_1 & 0 \\  0 & 0 \end{matrix} \right)$ and some idempotent $Q$ on the Banach space $X$. 
\end{proof}

We say that a compact subset $K$  of the complex plane {\it does not separate $0$ from $\infty$} if $0$ lies in the unbounded component of the complement of $K$. 

\begin{corollary}
Let $$ T = \left(\begin{matrix} 0 & B \\  C & 0 \end{matrix} \right) $$ be an operator on the Banach space  $X = X_1 \oplus X_2$.
Suppose that $\sigma(B C + \frac14 I_1)$ does not separate $0$ from $\infty$. Then $T$ is a commutator of the idempotent $P = \left(\begin{matrix} I_1 & 0 \\  0 & 0 \end{matrix} \right)$ and some idempotent $Q$ on $X$.
\label{separate}
\end{corollary}

\begin{proof}
By the Riesz functional calculus, the assumption implies that $B C + \frac14 I_1$ admits a square root 
$S_1$ which lies in the norm-closed algebra generated by $B C + \frac14 I_1$.
More precisely, the hypothesis implies that there is a function $f$, holomorphic in a simply connected open set $\Omega$ containing the closed set $\sigma(B C + \frac14 I_1) \cup\{ \frac14 \}$, which satisfies $(f(z))^2 = z$. 
It is well-known that the spectra $\sigma(B C + \frac14 I_1)$ and $\sigma(C B + \frac14 I_2)$ differ perhaps only in the point $\frac14$. So, we can define 
$S_1 = f(B C + \frac14 I_1)$ and $S_2 = f(C B + \frac14 I_2)$. 
Since $(B C + \frac14 I_1) B = B (C B+ \frac14 I_2)$, we have $S_1 B = B S_2$. Similarly, since $C (B C + \frac14 I_1) = (C B+ \frac14 I_2) C$, it holds that $C S_1 = S_2 C$.
By \Cref{operators},  $T$ is a commutator of the idempotent $P = \left(\begin{matrix} I_1 & 0 \\  0 & 0 \end{matrix} \right)$ and some idempotent $Q$ on the Banach space $X$. 
\end{proof}

As a special case of \Cref{separate} we obtain the following generalization of \cite[Lemma 5.6]{MRZ23}.

\begin{corollary}
Let $$ T = \left(\begin{matrix} 0 & B \\  C & 0 \end{matrix} \right) $$ be an operator on the Banach space  $X = X_1 \oplus X_2$.
Suppose that $r(B C) < \frac14$, where $r$ denotes the spectral radius function. Then $T$ is a commutator of the idempotent $P = \left(\begin{matrix} I_1 & 0 \\  0 & 0 \end{matrix} \right)$ and some idempotent $Q$ on $X$.
\label{14}
\end{corollary}

\section{Multiples of unilateral shifts on $l^p$-spaces}

Let us begin with the following simple lemma.

\begin{lemma}
Let $A$ be an operator on a Banach space $X$ such that $\dim (\ker A) = 1$ and $\dim (\ker A^2) = 2$. Then $A$ is not a square of some operator $B$ on $X$.
\label{not_square}
\end{lemma} 

\begin{proof}
Assume that $A = B^2$ for some operator $B$ on $X$. Then $\ker B \subseteq \ker A$, and so either $\ker B = \{0\}$ or $\ker B =\ker A$. 
If $\ker B = \{0\}$, then $\ker A =\ker B^2 = \{0\}$ that is not true. 
Therefore, $\ker B =\ker A= \ker B^2$. Then $\ker B^k =\ker B$ for all positive integers $k$, by \cite[Lemma 2.19]{AA02}. In particular, 
$\ker B = \ker B^4 = \ker A^2$. This contradicts the assumption that  $\dim (\ker A^2) = 2$. 
 \end{proof}

Let $S$ be the unilateral forward shift on either $l^p$ ($1 \le p \le \infty$) or $c_0$, that is, the operator defined by 
$S(x_1, x_2, x_3, \ldots ) = (0, x_1, x_2, x_3, \ldots )$. 
In \cite[Corollary 5.9]{MRZ23} it is shown that when $p=2$ the operator $\mu S$ is a commutator of idempotents if and only if the complex number $\mu$ satisfies $|\mu| \le \frac12$.
We now consider this in our context. To end this, we first write $S$ in the block form $\left(\begin{matrix} 0 & * \\  * & 0 \end{matrix} \right)$.
Let $\{e_n\}_{n}$ denote the standard unit vectors in either $l^p$ or $c_0$, and write $X_1 = \bigvee_n \{e_{2 n}\}$ and $X_2 = \bigvee_n \{e_{2 n-1}\}$. Then $X = X_1 \oplus X_2$, and with respect to this decomposition $S$ has the form 
$$ S_0 := \left(\begin{matrix} 0 & I \\  S & 0 \end{matrix} \right) . $$

\begin{theorem}
Let $\mu$ be a complex number. Then $\mu S_0$ is a commutator of the idempotent $P = \left(\begin{matrix} I_1 & 0 \\  0 & 0 \end{matrix} \right)$ and some idempotent $Q$ on either $l^p$ ($1 \le p < \infty$) or $c_0$
if and only if $|\mu| \le \frac12$.
\label{shift}
\end{theorem}

\begin{proof}
Since $S$ is similar to $\alpha S$ via a diagonal operator whenever $\alpha \in \{ z \in \mathbb C: |z| = 1\}$ we may assume without loss of generality that $\mu \ge 0$. 
We consider $3$ cases.

{\it Case 1}. If $0 \le \mu < \frac12$, then $r(\mu S \cdot \mu I) < \frac14$, and so by \Cref{14}  $\mu S_0$ is a commutator of the idempotent $P = \left(\begin{matrix} I_1 & 0 \\  0 & 0 \end{matrix} \right)$ and some idempotent $Q$.

{\it Case 2}. If $\mu = \frac12$, then by \Cref{operators} we must show that the operator $\frac14 S + \frac14 I = \frac14 (S+I)$ has a square root commuting with $S$. 
           It is well-known that $\sum_{n=0}^\infty |{ \frac12 \choose n} | < \infty$, and so 
           $R := \sum_{n=0}^\infty { \frac12 \choose n} S^n$ converges absolutely, $R S = S R$ and $R^2 = S+I$. 

{\it Case 3}. If $\mu > \frac12$, then put $\lambda := - \frac{1}{4 \mu^2} \in (-1,0)$, so that $\mu^2 S + \frac14 I = \mu^2(S-\lambda I)$. 
           In view of \Cref{operators} it is enough to show that $S-\lambda I$ has no square root on either $l^p$  ($1 \le p < \infty$) or $c_0$. 
            Then it is enough to prove that the adjoint operator $A := S^* - \lambda I$ has no square root on either the dual space $(l^p)^* = l^q$, where $q$ is the conjugate exponent of $p$, or the dual space $c_0^* = l^1$.
           To end this, we apply \Cref{not_square}. It is easy to show that 
           $$ \ker A = \bigvee \{(1, \lambda, \lambda^2, \lambda^3, \ldots, )\}  \ \ \ {\rm and} $$
           $$ \ker A^2 = \bigvee \{(1, \lambda, \lambda^2, \lambda^3, \ldots ) , (0, 1, 2 \lambda, 3 \lambda^2, 4 \lambda^3, \ldots )\} , $$
           so that $\dim (\ker A) = 1$ and $\dim (\ker A^2) = 2$. This completes the proof.
\end{proof}

In a similar manner one can prove the corresponding result for the unilateral backward shift $B$ on  either $l^p$ ($1 \le p \le \infty$) or $c_0$, that is, the operator defined by 
$B(x_1, x_2, x_3, \ldots ) = (x_2, x_3, x_4, \ldots )$. We will omit its proof. 

As before,  we first write $B$ in the block form $\left(\begin{matrix} 0 & * \\  * & 0 \end{matrix} \right)$.
Let $\{e_n\}_{n}$ denote the standard unit vectors in either $l^p$ or $c_0$, and write $X_1 = \bigvee_n \{e_{2 n}\}$ and $X_2 = \bigvee_n \{e_{2 n-1}\}$. Then $X = X_1 \oplus X_2$, and with respect to this decomposition $B$ has the form 
$$ B_0 := \left(\begin{matrix} 0 & B \\  I & 0 \end{matrix} \right) . $$

\begin{theorem}
Let $\mu$ be a complex number. Then $\mu B_0$ is a commutator of the idempotent $P = \left(\begin{matrix} I_1 & 0 \\  0 & 0 \end{matrix} \right)$ and some idempotent $Q$ on either $l^p$ ($1 \le p \le \infty$) or $c_0$
if and only if $|\mu| \le \frac12$.
\label{shift}
\end{theorem}

\subsection*{Acknowledgement}
The author acknowledges the financial support from the
  Slovenian Research Agency (research core funding No. P1-0222).


\end{document}